\documentclass[12pt]{amsart}
\usepackage{amssymb,amsmath,cite}
\usepackage[colorlinks]{hyperref}
\usepackage[capitalize]{cleveref}

\newtheorem{thm}{Theorem}[section]
\newtheorem{lem}[thm]{Lemma}
\newtheorem{prop}[thm]{Proposition}
\theoremstyle{definition}
\newtheorem{dfn}[thm]{Definition}
\newtheorem{eg}[thm]{Example}
\newtheorem{conj}[thm]{Conjecture}
\numberwithin{equation}{section}
\def\e{\mathrm{e}}
\def\o{\mathrm{o}}
\def\={\;=\;}
\def\gge{\;\ge\;}
\def\lle{\;\le\;}
\def\rmand{\qquad\text{and}\qquad}
\def\psim{\sim_{\mathrm{p}}}
\def\wsim{\sim_{\mathrm{w}}}
\def\LC{\mathcal{L}}

\crefname{def}{Def.}{Defs.}
\Crefname{def}{Def.}{Defs.}
\creflabelformat{def}{~\upshape(#2#1#3)}

\crefname{thm}{Theorem}{Theorems}
\Crefname{thm}{Theorem}{Theorems}
\creflabelformat{thm}{~\upshape #2#1#3}

\crefname{eg}{Example}{Examples}
\Crefname{eg}{Example}{Examples}
\creflabelformat{eg}{~\upshape #2#1#3 }

\crefname{ineq}{Ineq.}{Ineqs.}
\Crefname{ineq}{Inequality}{Inequalities}
\creflabelformat{ineq}{~\upshape(#2#1#3)}

\title[Partial Synchronicity of Log-concave Sequences]
{Convolution Preserves Partial Synchronicity of Log-concave Sequences}

\author[H. Hu]{H. Hu}
\address{School of Mathematical Sciences, \& LMAM\\
Peking University, 100871 Beijing, P. R. China; 
email: huhan@pku.edu.cn}
\author[D.G.L. Wang]{David G.L. Wang$^\dag$$^\ddag$}
\address{
$^\dag$School of Mathematics and Statistics, Beijing Institute of Technology, 102488 Beijing, P. R. China\\
$^\ddag$Beijing Key Laboratory on MCAACI, Beijing Institute of Technology, 102488 Beijing, P. R. China\\
email: glw@bit.edu.cn}
\author[F. Zhao]{F. Zhao}
\address{School of Mathematical Sciences, \& LMAM\\
Peking University, 100871 Beijing, P. R. China; 
email: zhf327@pku.edu.cn}
\author[T.Y. Zhao]{T.Y. Zhao}
\address{School of Mathematical Sciences, \& LMAM\\
Peking University, 100871 Beijing, P. R. China; 
email: zhaotongyuan@pku.edu.cn}


\begin{document}

\begin{abstract}   
In a recent proof of the log-concavity of genus polynomials of some families of graphs, Gross et al.\ defined the weakly synchronicity relation between log-concave sequences, and conjectured that the convolution operation by any log-concave sequence preserves weakly synchronicity. We disprove it by providing a counterexample. Furthermore, we find the so-called partial synchronicity relation between log-concave sequences, which is (i) weaker than the synchronicity, (ii) stronger than the weakly synchronicity, and (iii) preserved by the convolution operation.
\end{abstract}
\maketitle

\section{\large Introduction}

The log-concavity of sequences of nonnegative numbers has been paid extensive and intensive attention 
during the past thirty years, see Stanley~\cite{Sta89} and Brenti~\cite{Bre89,Bre94}.
In the late 1980s, Gross et al.~\cite{GRT89} posed the LCGD conjecture that the genus polynomial of every graph is log-concave,
which firstly connected the log-concavity of sequences with topological graph theory, 
or more precisely, with the surface embedding of graphs.
For survey books of topological graph theory, see~\cite{GT01B,BW09B}.
In the recent work~\cite{GMTW15}, Gross et al.\ established a criterion determining the log-concavity 
of sum of products of log-concave polynomials. 
With aid of the criterion, they confirmed the LCGD conjecture for several families of 
graphs generated by vertex- or edge-amalgamations, including the graphs called iterated $4$-wheels.

The criterion is considered to have its own interest,
since it deals with the intrinsic arithmetic relations between log-concave polynomials.
See~\cite{BBL09,Men69,Lig97} for related papers.
The idea of the criterion consists of three key parts, the synchronicity, the radio-dominance,
and the lexicographicity. 
It is the synchronicity part, 
which originally arises from common facts observed from topological embeddings of graphs into surfaces, 
starts the whole development of the new log-concave results.

Though the synchronicity relation is sufficient to judge the log-concavity of positive linear combination of log-concave polynomials,
Gross et~al.\ managed to weaken it to certain weakly synchronicity relation. 
The first power of such a weaker relation was supposed to be preserved by sequence convolution,
which was posed as the following conjecture; see \cite[Conjecture~2.13]{GMTW15}.

\begin{conj}\label{conj}
Let $A$, $B$, $C$ be three log-concave nonnegative sequences without internal zeros. 
If $A\wsim B$, then the convolution sequences $A*C$ and $B*C$ are weakly synchronized.
\end{conj}

We disprove \cref{conj} by providing an explicit counterexample.
This example leads us to find a relation in \cref{def:ps}, called partial synchronicity, 
between log-concave sequences, to achieve the original motivation. 
Namely, the partial synchronicity relation is (i) weaker than synchronicity, (ii) stronger than the weakly synchronicity,
and (iii) preserved by the convolution operation.
See \cref{thm:ws<ps<s,thm:main}.

\section{Preliminary and the Counterexample}

All sequences concerned in the paper consists of nonnegative numbers.
For any finite sequence $A=(a_k)_{k=0}^n$ of nonnegative numbers,
we identify the sequence $A$ with the infinite sequence $(a_k')_{k\in\mathbb{Z}}$,
where $a_k'=a_k$ for $0\le k\le n$, and $a_k'=0$ otherwise.  
Under this convenience, one may denote $A=(a_k)$ for simplicity.
We write $uA$ to denote the scalar multiple sequence $(ua_k)$, for any constant $u\ge0$.   
Let $B=(b_k)$ be another sequence of nonnegative numbers.
Then the notation $A+B$ stands for the sequence $(a_k+b_k)$.  

We call the first positive term of the sequence $A$ the {\em head} of~$A$,
and call the last positive term the {\em tail} of~$A$.
In other words, the term~$a_h$ is said to be the head of~$A$ if $a_{h-1}=0<a_h$.
In this case, we call the integer~$h$ the {\em head index} of~$A$, denoted \hbox{$h(A)=h$}.
Similarly, one may define the {\em tail index}, denoted as~$t(A)$.
It is clear that $h(A)\le t(A)$.
Without loss of generality, we suppose that $h(A)\ge 0$ for all sequences concerned in this paper.

The sequence $A$ is said to be {\em log-concave} if 
$a_k^2\ge a_{k-1}a_{k+1}$ for all integers $k$.
It is said to {\em have no internal zeros} if for any integers $i<j$ such that $a_ia_j>0$, one has $\prod_{k=i}^{j}a_k>0$.
Denote by $\LC$ the set of log-concave sequences without internal zeros.
We call the sequence consisting of only zeros the {\em zero sequence}, denoted~$(0)$.
Denote 
\[
\LC^*=\LC\setminus\{(0)\}.
\]

\begin{dfn}\label{def:sync}
Let $A=(a_k)\in\LC$ and $B=(b_k)\in\LC$.
We say that the sequences $A$ and $B$ are {\em synchronized}, denoted as $A\sim B$, 
if
\[
a_{k-1}b_{k+1}\,\le\,a_kb_k \rmand a_{k+1}b_{k-1}\,\le\,a_kb_k
\]
for all $k$.
\end{dfn}
It is obvious that scalar multiplications preserve synchronicity.  
Moreover, the synchronicity relation is reflexive, symmetric and non-transitive; see~\cite{GMTW15}.

\begin{dfn}\label{def:ws}
Let $A=(a_k)\in\LC$ and $B=(b_k)\in\LC$.
We say that the sequences $A$ and $B$ are {\em weakly synchronized}, denoted $A\wsim B$, 
if
\begin{equation}\label[ineq]{ineq:def:ws}
a_{k-1}b_{k+1}+a_{k+1}b_{k-1}\lle 2a_kb_k, 
\end{equation}
for all $k$.
\end{dfn}

For example, consider the sequences $A=(1,3,5)$ and $B=(1,4,13)$. It is easy to verify
that $A\wsim B$ and $A\not\sim B$.

Recall that if $A=(a_k)_{k=0}^m$ and $B=(b_h)_{h=0}^n$,
the convolution sequence $A*B$ is defined to be 
the coefficient sequence of the polynomial product
\[
\left(\sum_{i=0}^ma_ix^i\right)\left(\sum_{i=0}^nb_jx^j\right).
\]
The next example disproves \cref{conj}.

\begin{eg}\label[eg]{eg:c}
Let
\begin{align*}
A&\=(\,1,\,20,\,200,\,1800\,),\\
B&\=(1,\,6,\,30,\,60\,),\\
C&\=(\,40,\,60,\,10,\,1\,).
\end{align*}
It is direct to verify $A\wsim  B$ from \cref{def:ws}, and to compute that 
\begin{align*}
A*C&\=(\,40,\,860,\,9210,\,84201,\,110020,\,18200,\,1800\,),\\
B*C&\=(\,40,\,300,\,1570,\,4261,\,3906,\,630,\,60\,).
\end{align*}
Then, for the convolution sequences $A*C$ and $B*C$, \cref{ineq:def:ws} does not hold for $k=2$:
\begin{align*}
&(A*C)_1(B*C)_3+(A*C)_3(B*C)_1-2(A*C)_2(B*C)_2\\
\=&860\times 4261+84201\times 300-2\times 9210\times 1570\\
\;>\;&0.
\end{align*}
\end{eg}

\section{The Partial Synchronicity Relation}

In this section, we introduce the partial synchronicity relation between log-concave sequences,
which is expected to serve the original motivation of Gross et al.\ in~\cite{GMTW15}.

Let $A=(a_k)$ and $B=(b_k)$ be two sequences of numbers. For any integers $m$ and $n$,
we define
\begin{equation}\label[def]{def:f}
f(A,B;\ m,n)\=a_mb_n+a_nb_m.
\end{equation}
When there is no confusion, we simply denote 
\[
f(m,n)\=f(A,B;\ m,n).
\]
From \cref{def:f}, we see that the function $f(m,n)$ is commutative, namely,
\begin{equation}\label{comm:f}
f(m,n)\=f(n,m)
\end{equation}
for all integers $m$ and $n$.
For further discussion, we need the following lemma.

\begin{lem}\label{lem:trl}
Suppose that
\begin{equation}\label[ineq]{cond:ws}
f(m,n)\gge f(m+1,\,n-1)
\end{equation}
for all integers $m$ and $n$ such that $m\ge n$.
Then we have 
\begin{equation}\label[ineq]{lem:dsr}
f(a,b)\gge f(c,d)
\end{equation}
for any integers $a,b,c,d$ such that 
\begin{align}
&a+b=c+d,\qquad\text{and that}\label{eq:abcd}\\[3pt]
&|a-b|<|c-d|.\label[ineq]{ineq:abcd}
\end{align}
\end{lem}

\begin{proof}
Let $a,b,c,d$ be integers satisfying \cref{eq:abcd,ineq:abcd}.
In order to show \cref{lem:dsr}, one may suppose, by the commutativity \cref{comm:f}, 
that $a\ge b$ and $c\ge d$. Then \cref{ineq:abcd} reduces to
\begin{equation}\label[ineq]{ineq:abcd2}
a-b<c-d.
\end{equation}
Summing up \cref{eq:abcd,ineq:abcd2}, one obtains that 
\begin{equation}\label[ineq]{a<c}
a\;<\;c.
\end{equation}

Substituting $m=a$ and $n=b$ in the premise \cref{cond:ws}, one finds that 
\begin{equation}\label[ineq]{ineq:ab1}
f(a,b)\gge f(a+1,\,b-1).
\end{equation}
Since $a\ge b$, we have $a+1\ge b-1$. 
Therefore, in \cref{ineq:ab1},
by replacing the number~$a$ by $a+1$, and replacing $b$ by $b-1$,
we obtain that
\begin{equation}\label[ineq]{ineq:pf2}
f(a+1,\,b-1)\gge f(a+2,\,b-2).
\end{equation}
The same substitution for \cref{ineq:pf2} gives that
\[
f(a+2,\,b-2)\gge f(a+3,\,b-3).
\]
Continuing in this way, one finds 
\begin{equation}\label[ineq]{ineq:pf3}
f(a+i-1,\,b-i+1)\gge f(a+i,\,b-i)
\end{equation}
for all positive integers $i$.
Since $a<c$ from \cref{a<c}, we can sum up \cref{ineq:pf3} over $i\in\{1,2,\ldots,c-a\}$, which yields that
\[
f(a,\,b)\gge f(c,\,b-c+a).
\]
Hence, we obtain the desired \cref{lem:dsr}, by noticing $d=b-c+a$ from \cref{eq:abcd}.
\end{proof}

\begin{dfn}\label{def:ps}
Let $A, B\in\LC$. We say that the sequences $A$ and~$B$ are {\em partially synchronized}, 
denoted by $A\psim B$, 
if \cref{cond:ws}, or equivalently,
\begin{equation}\label[ineq]{ineq:def:ps}
a_mb_n+a_nb_m\gge a_{m+1}b_{n-1}+a_{n-1}b_{m+1},
\end{equation}
holds for all integers $m$ and $n$ such that $m\ge n$.
\end{dfn}

It is clear that scalar multiplications preserve partial synchronicity.  
Moreover, the partial synchronicity relation is reflexive, symmetric, and non-transitive.
The non-transitivity can be seen from the example
\[
A=(\,1,\,2,\,3\,),\qquad
B=(\,1,\,3,\,8\,),\qquad
C=(\,1,\,4,\,15\,),
\]
where $A\psim B$, $B\psim C$, and $A\not\psim C$.
In fact, this above example has been used to exemplify the non-transitivity of the synchronicity relation in \cite{GMTW15}.

The next proposition helps check quickly the weakly synchronicity of two sequences,
which is also of help in the proof of \cref{thm:ws<ps<s}.

\begin{prop}\label{prop:easycheck}
Let $A,B\in\LC^*$. Then $A\psim B$ holds iff 

\vskip 2pt\noindent
(i) $|h(A)-h(B)|\le1$;
\vskip 2pt\noindent
(ii) $|t(A)-t(B)|\le 1$; and
\vskip 2pt\noindent
(iii) \cref{ineq:def:ps} holds for all integers $m$ and $n$ such that $m\ge n$,
\[
\max\{h(A),h(B)\}\lle m\lle \max\{t(A),t(B)\}-1,
\] 
and that
\[
\min\{h(A),h(B)\}+1\lle n\lle \min\{t(A),t(B)\}.
\] 
\end{prop}

\begin{proof}
Let $A=(a_k)$ and $B=(b_k)$ be sequences such that $A,B\in\LC^*$.

\medskip
\noindent{\bf Necessity.} Suppose that $A\psim B$, i.e., \cref{ineq:def:ps} holds for all integers~$m$ and $n$ such that $m\ge n$.

In order to show (i), one may suppose that $h(A)\le h(B)$ without loss of generality.
Assume that $|h(A)-h(B)|\ge 2$. 
Take 
\[
m=h(B)-1
\rmand
n=h(A)+1.
\]
It follows that $m\ge n$, and that \cref{ineq:def:ps} becomes
\begin{equation}\label[ineq]{pf8}
a_{h(B)-1}b_{h(A)+1}+a_{h(A)+1}b_{h(B)-1}
\ge 
a_{h(B)}b_{h(A)}+a_{h(A)}b_{h(B)}.
\end{equation}
From definition of the head $h(B)$, we have 
\begin{equation}\label{pf11}
b_{h(B)-1}=0.
\end{equation}
On the other hand, since $h(B)-h(A)\ge 2$, we have 
\begin{equation}\label{pf10}
b_{h(A)}=0
\rmand
b_{h(A)+1}=0.
\end{equation}
Substituting \cref{pf10,pf11} into \cref{pf8},
one obtains 
\begin{equation}\label[ineq]{pf9}
0\gge a_{h(A)}b_{h(B)}.
\end{equation}
From the definition of the head function $h$,
one sees that $a_{h(A)}>0$ and $b_{h(B)}>0$,
contradicting \cref{pf9}.

Condition (ii) can be shown along the same lines. 
The necessity of (iii) is obvious from the premise $A\psim B$.

\medskip
\noindent{\bf Sufficiency.}
For convenience, we denote
\begin{align*}
mh=\min\{h(A),h(B)\},\,&\qquad\,
mt=\min\{t(A),t(B)\},\\[4pt]
Mh=\max\{h(A),h(B)\},&\qquad
Mt=\max\{t(A),t(B)\}.
\end{align*}
If $m\ge Mt$, then $a_{m+1}=b_{m+1}=0$, and thus 
\begin{equation}\label{pf18}
f(m+1,\,n-1)\=0.
\end{equation}
Since $f(m,n)\ge0$, \cref{cond:ws} holds.
Below we can suppose that 
\begin{equation}\label[ineq]{pf19}
m\lle Mt-1.
\end{equation}
In another case that $n\le mh$, we have $a_{n-1}=b_{n-1}=0$.
Therefore, we infer \cref{pf18}, which allows us to suppose without loss of generality that 
\begin{equation}\label[ineq]{pf20}
n\gge mh+1.
\end{equation}

In view of (iii), \cref{pf19,pf20}, 
it suffices to prove \cref{cond:ws} for all integers~$m$ and $n$ such that $m\ge n$, and that 
either $m\le Mh-1$ or $n\ge mt+1$.

When $m\le Mh-1$, by using \cref{pf20}, one may deduce that 
\[
Mh-1\gge m\gge n\gge mh+1,
\]
contradicting Condition (i), which implies that $Mh-mh\le 1$.
When $n\ge mt+1$, by using \cref{pf19}, we can derive that 
\[
mt+1\lle n\lle m\lle Mt+1,
\]
contradicting Condition (ii), which implies that $Mt-mt\le 1$.
This completes the proof.
\end{proof}

Now we can clarify the relations among the synchronicity,
the weak synchronicity, and the partial synchronicity.

\begin{thm}\label{thm:ws<ps<s}
The partial synchronicity relation $\psim$ is weaker than synchronicity~$\sim$,
and stronger than the weakly synchronicity~$\wsim$. In other words,
any two synchronized log-concave sequences without internal zeros are partially synchronized,
and any two partially synchronized log-concave sequences without internal zeros are weakly synchronized.
\end{thm}

\begin{proof}
Taking $m=n=k$ in \cref{ineq:def:ps} gives \cref{ineq:def:ws},
which implies that the partial synchronicity~$\psim$ is stronger than weakly synchronicity~$\wsim$.

Let $A=(a_k)\in\LC$ and $B=(b_k)\in\LC$. 
Let $m\ge n$. It suffices to show \cref{ineq:def:ps}.
By \cref{prop:easycheck} (iii), we can suppose that 
\[
m\lle \max\{t(A),\,t(B)\}-1
\rmand
n\gge\min\{h(A),\,h(B)\}+1.
\]
Thus we have $a_{m+1}b_{m+1}\ne 0$ and $a_nb_n\ne0$. 
Since $A,B\in\LC$, neither of the sequences $A$ and $B$ has internal zeros.
It follows that 
\[
\prod_{i=n}^{m+1}(a_ib_i)\ne 0.
\]
By dividing \cref{ineq:def:ps} by the factor $a_{m+1}b_{m+1}$, 
we see that it is equivalent to prove
\begin{equation}\label[ineq]{dsr:100}
\frac{a_mb_n}{a_{m+1}b_{m+1}}+\frac{a_nb_m}{a_{m+1}b_{m+1}}
\;\ge\;
\frac{b_{n-1}}{b_{m+1}}+\frac{a_{n-1}}{a_{m+1}}.
\end{equation}
Following the notation in \cite{GMTW15}, we let
\[
\alpha_k=\frac{a_k}{a_{k-1}}
\rmand
\beta_h=\frac{b_h}{b_{h-1}},
\]
when $a_{k-1}\ne0$ and $b_{h-1}\ne0$.
Then the desired \cref{dsr:100} can be recast as
\begin{equation}\label[ineq]{dsr:200}
\frac{1}{\alpha_{m+1}}\prod_{i=n+1}^{m+1}\frac{1}{\beta_i}
+\frac{1}{\beta_{m+1}}\prod_{i=n+1}^{m+1}\frac{1}{\alpha_i}
\;\ge\;
\prod_{i=n}^{m+1}\frac{1}{\beta_i}+\prod_{i=n}^{m+1}\frac{1}{\alpha_i}.
\end{equation}
Multiplying \cref{dsr:200} by the product $\prod_{i=n}^{m+1}(\alpha_i\beta_i)$,
we find to show the following inequality is sufficient:
\[
\beta_n\prod_{i=n}^m\alpha_{i}+\alpha_n\prod_{i=n}^m\beta_{i}
\;\ge\;
\prod_{i=n}^{m+1}\alpha_i+\prod_{i=n}^{m+1}\beta_i.
\]
That is, it suffices to show that 
\begin{equation}\label[ineq]{dsr:300}
(\beta_n-\alpha_{m+1})\prod_{i=n}^m\alpha_{i}+
(\alpha_n-\beta_{m+1})\prod_{i=n}^m\beta_{i}
\;\ge\;0.
\end{equation}
On the other hand, the synchronicity relation $A\sim B$ implies that 
\[
\alpha_n\ge \beta_{n+1}
\rmand
\beta_n\ge \alpha_{n+1}.
\]
By the log-concavity of the sequence $B$, the sequence $\beta_k$ is decreasing. Thus we have
\begin{equation}\label[ineq]{pf:100}
\alpha_n\ge\beta_{n+1}\ge\beta_{m+1}.
\end{equation}
For the same reason, we have 
\begin{equation}\label[ineq]{pf:200}
\beta_n\ge\alpha_{m+1}.
\end{equation}
In view of \cref{pf:100,pf:200}, the desired \cref{dsr:300} follows immediately.
This completes the proof.
\end{proof}

Gross et al.~\cite[Theorems 2.10, 2.11]{GMTW15} showed that 
any collection of pairwise synchronized sequences is closed under linear combinations with nonnegative coefficients,
and that the same property holds for the weak synchronicity relation.
We show that partial synchronicity behaves in the same manner in \cref{lem:WS:LC,thm:lincomb}.

\begin{lem}\label{lem:WS:LC}
Let $A,B\in\LC$ such that $A\psim B$.
Then we have $uA+vB\in\LC$ for all nonnegative numbers $u$ and $v$.
\end{lem}

\begin{proof}
Let $A=(a_k)$ and $B=(b_k)$ be log-concave sequences such that $A\psim B$.
Let $u,v\ge 0$.
Since the sequence $A$ is log-concave, we have 
\begin{equation}\label[ineq]{pf:LCA}
u^2a_k^2\gge u^2a_{k-1}a_{k+1}. 
\end{equation}
For the same reason, the log-concavity of the sequence $B$ implies that 
\begin{equation}\label[ineq]{pf:LCB}
v^2b_k^2\gge v^2b_{k-1}b_{k+1}. 
\end{equation}
Since $A\psim B$,
one may take $m=n=k$ in \cref{ineq:def:ps}, which yields
\begin{equation}\label[ineq]{pf:wsAB}
uv(a_kb_k+a_kb_k)\gge uv(a_{k+1}b_{k-1}+a_{k-1}b_{k+1}).
\end{equation}
Adding \cref{pf:LCA,pf:LCB,pf:wsAB} up, we obtain that
\[
(ua_k+vb_k)^2\gge (ua_{k+1}+vb_{k+1})(ua_{k-1}+vb_{k-1}).
\]
In other words, the sequence $uA+vB$ is log-concave.
\end{proof}

\begin{thm}\label{thm:lincomb}
Suppose that the sequences $A_1,A_2,\ldots,A_n$ are pairwise partially synchronized.
Then for any nonnegative numbers $u_1$, $v_1$, $u_2$, $v_2$, $\ldots$, $u_n$, $v_n$, we have
$\sum_{i=1}^n u_i A_i\psim\sum_{i=1}^n v_i A_i$.
\end{thm}

\begin{proof} 
Since scalars preserve the weakly synchronicity relation, 
we see that the $2n$ sequences $u_iA_i$ and~$v_iA_i$ are pairwise partially synchronized.  
By iterative application, it suffices to show that summation preserves partial synchronicity.
Namely, given $A,B,C\in\LC^*$,
we only need to show that $A+B\psim C$ if $A\psim C$ and $B\psim C$.

Let $m$ and $n$ be integers such that $m\ge n$.
The condition $A\psim C$ implies that
\begin{equation}\label[ineq]{pf2:AC}
a_mc_n+a_nc_m\gge a_{m+1}c_{n-1}+a_{n-1}c_{m+1}.
\end{equation}
The condition $B\psim C$ implies that 
\begin{equation}\label[ineq]{pf2:BC}
b_mc_n+b_nc_m\gge b_{m+1}c_{n-1}+b_{n-1}c_{m+1}.
\end{equation}
Adding \cref{pf2:AC,pf2:BC} up, one obtains that
\[
(a_m+b_m)c_n+(a_n+b_n)c_m\gge (a_{m+1}+b_{m+1})c_{n-1}+(a_{n-1}+b_{n-1})c_{m+1}.
\]
On the other hand, the sequence $A+B$ is log-concave by \cref{lem:WS:LC}.
Hence, we find $A+B\psim C$.
This completes the proof.
\end{proof}

In \cite[Theorem 2.12]{GMTW15}, Gross et al.\ also showed that the synchronicity relation 
is preserved by the sequence convolution operation.
\Cref{eg:c} illustrates that this property does not hold for weak synchronicity.
Below we demonstrate that the same property holds for partial synchronicity.

\begin{thm}\label{thm:main}
Let $A,B,C\in\LC^*$. If $A\psim B$,
then $A*C\psim B*C$.
\end{thm}

\begin{proof}
Suppose that $A=(a_k)$, $B=(b_k)$, and $C=(c_k)$.
Since all the sequences $A$, $B$ and $C$ are log-concave without internal zeros,
so are the sequences $A*C$ and $B*C$.
Let~$m$ and $n$ be integers such that $m\ge n$.
From \cref{def:ps} of weakly synchronization, it suffices to show that
\[
f(A*C,\,B*C;\ m,\,n)\ge f(A*C,\,B*C;\ m+1,\,n-1).
\]
that is,
\begin{multline}\label[ineq]{ineq:dsr:main}
(A*C)_m(B*C)_n+(A*C)_n(B*C)_m\\
\gge (A*C)_{m+1}(B*C)_{n-1}+(A*C)_{n-1}(B*C)_{m+1},
\end{multline}
where the notation $S_n$ for a sequence $S$ denotes the $n$th term of $S$.

We consider each summand in \cref{ineq:dsr:main} as a linear combination of the products of form $c_kc_l$,
where $k<l$ are integers. 
Then the coefficient of $c_kc_l$ in the expansion of the first summand $(A*C)_m(B*C)_n$ is 
\[
a_{m-k}b_{n-l}+a_{m-l}b_{n-k}.
\]
Dealing with the second summand $(A*C)_n(B*C)_m$ by exchanging the numbers $m$ and $n$ in the above expression, 
we find the coefficient of $c_kc_l$ of 
the left hand side of \cref{ineq:dsr:main} is
\begin{multline}\label{pf:coeffL}
(a_{m-k}b_{n-l}+a_{m-l}b_{n-k})+(a_{n-k}b_{m-l}+a_{n-l}b_{m-k})\\
\=f(m-k,\,n-l)+f(m-l,\,n-k).
\end{multline}
In \cref{pf:coeffL}, replacing $m$ by $m+1$, and replacing $n$ by $n-1$,
we find that the coefficient of $c_kc_l$ of the right hand side of \cref{ineq:dsr:main} is
\[
f(m+1-k,\,n-1-l)+f(m+1-l,\,n-1-k).
\]
In the same way, one may check that the coefficients of $c_k^2$ in the two sides of \cref{ineq:dsr:main} 
are respectively
\[
f(m-k,\,n-k)\rmand f(m+1-k,\,n-1-k).
\]
To sum up, we can recast the desired \cref{ineq:dsr:main} in terms of the function $f$ as 
\begin{align}
&\sum_{k<l}\,[f(m-k,\,n-l)+f(m-l,\,n-k)]\cdot c_kc_l\label[ineq]{dsr1}\\
-\;&\sum_{k<l}\,[f(m+1-k,\,n-1-l)+f(m+1-l,\,n-1-k)]\cdot c_kc_l\notag\\
+\;&\sum_{k}\,[f(m-k,\,n-k)-f(m+1-k,\,n-1-k)]\cdot c_k^2\gge 0,\notag
\end{align}
where the indices of every summation run over, in fact, a finite number of integers
(since the number of non-zero terms in the sequence $C$ is finite).
We omit the range of such indices, and adopt this simplicity convention throughout this paper.

We define
\begin{equation}\label[def]{def:g}
g(k,l)\=f(m-k,\,n-l)-f(m+1-l,\,n-1-k).
\end{equation}
Then the desired \cref{dsr1} can be written simply as
\[
\sum_{k<l}[g(k,l)+g(l,k)]c_kc_l
+\sum_kg(k,k)c_k^2\gge 0,
\]
that is,
\begin{equation}\label[ineq]{dsr2}
\sum_{k,\,l}g(k,l)c_kc_l\gge 0.
\end{equation}
Let $s$ be a nonnegative integer, indicating the sum $k+l$ of the indices.
For notation simplicity, we define
\begin{align}
h_\e(k)&\=g(k,\,2s-k),\label[def]{def:h:even}\\[3pt]
h_\o(k)&\=g(k,\,2s+1-k),\label[def]{def:h:odd}
\end{align}
where the subscript letter ``e'' indicates that the sum $2s$ of the two variates $k$ and $2s-k$ in \cref{def:h:even} is an even integer,
and the subscript letter ``o'' indicates ``odd''.
By virtue of these notation, the desired \cref{dsr2} can be recast as
\[
\sum_{s,\,k}[h_\e(k)c_kc_{2s-k}+h_\o(k)c_kc_{2s+1-k}]\gge 0.
\]
Thus, it suffices to show, for all integers $s\ge 0$, that 
\begin{equation}\label[ineq]{dsr:h:even}
\sum_kh_\e(k)c_kc_{2s-k}\gge 0
\end{equation}
and that
\begin{equation}\label[ineq]{dsr:h:odd}
\sum_kh_\o(k)c_kc_{2s+1-k}\gge 0.
\end{equation}
We shall show them individually. Let $s\ge 0$.

We transform the left hand sides of the desired \cref{dsr:h:even} as
\begin{equation}\label{dsr1:even}
\sum_kh_\e(k)c_kc_{2s-k}
\=\sum_{k=0}^{s}(c_{s-k}c_{s+k}-c_{s-k-1}c_{s+k+1})\!\!\!\!\sum_{i=s-k-1}^{s+k+1}h_\e(i).
\end{equation}
Since the sequence $C$ is log-concave, we infer that 
\[
c_{s-k}c_{s+k}-c_{s-k-1}c_{s+k+1}\gge0.
\]
Thus, in view of \cref{dsr1:even}, the desired \cref{dsr:h:even} holds if 
\[
\sum_{i=s-k-1}^{s+k+1}h_\e(i)\gge0\qquad\text{for all $0\le k\le s$}.
\]

Now, let us reduce the left hand side $\sum_{i=s-k-1}^{s+k+1}h_\e(i)$.
From \cref{def:g} of the function~$g$, it is straightforward to verify that 
\begin{equation}\label{eq:g}
g(k,l)+g(l-1,\,k+1)=0.
\end{equation}
Taking $k=s-i$ and $l=s+i$, \cref{eq:g} becomes
\begin{equation}\label{pf:10}
g(s-i,\,s+i)+g(s+i-1,\,s-i+1)\=0.
\end{equation}
By \cref{def:h:even} of the function $h_\e$, \cref{pf:10} can be rewritten as
\begin{equation}\label{eq:h:even}
h_\e(s-i)+h_\e(s+i-1)\=0.
\end{equation}
By using \cref{eq:h:even}, we can simplify
\begin{align*}
\sum_{i=s-k-1}^{s+k+1}h_\e(i)&\=h_\e(s+k+1).
\end{align*}
Thereby to confirm the desired \cref{dsr:h:even}, 
it suffices to show $h_\e(s+k+1)\ge 0$, that is, $g(s+k+1,\,s-k-1)\ge 0$.
To do this, we will prove a stronger result that 
\begin{equation}\label[ineq]{ineq:g}
g(k,l)\ge 0\qquad\text{for all $k\ge l$}.
\end{equation}
On the way using  \cref{lem:trl}, one needs to check three conditions.
First, the sequences~$A$ and $B$ are partially synchronized as in the premise.
Second, the sum of variates of the functions 
\[
f(m-k,\,n-l)
\rmand
f(m+1-l,\,n-1-k)
\] 
are equal, i.e.,
\[
(m-k)+(n-l)\=(m+1-l)+(n-1-k).
\]
Third, since $m\ge n$ and $k\ge l$, the distances of variates are comparative as 
\[
|(m-k)-(n-l)|\lle m-n\;<\;|(m+1-l)-(n-1-k)|.
\]
By \cref{lem:trl}, we deduce the claimed \cref{ineq:g}.

When $k\ge 0$, we have $s+k+1\ge s-k-1$, and hence by \cref{ineq:g},
\[
h_\e(s+k+1)\=g(s+k+1,\,s-k-1)\gge0.
\]
This completes the proof of \cref{dsr:h:even}.

\Cref{dsr:h:odd} can be shown along the same lines.
In fact,
\[
\sum_kh_\o(k)c_kc_{2s+1-k}
\=\sum_{k=0}^s(c_kc_{2s+1-k}-c_{k-1}c_{2s+2-k})\!\!\sum_{i=k}^{2s+1-k}h_\o(i)\label{dsr1:odd}.
\]
Since the sequence $C$ is log-concave, we have
\[
c_kc_{2s+1-k}-c_{k-1}c_{2s+2-k}\gge 0.
\]
Thus the desired \cref{dsr:h:even} holds if 
\[
\sum_{i=k}^{2s+1-k}h_\o(i)\gge0\qquad\text{for all $0\le k\le s$}.
\]
On the other hand, from \cref{def:h:odd} of the function $h_\o$, \cref{eq:g} implies that
\begin{equation}\label{eq:h:odd}
h_\o(i)+h_\o(2s-i)=0.
\end{equation}
In particular, taking $i=s$ in \cref{eq:h:odd}, one obtains that
\begin{equation}\label{eq:h:odd:s}
h_\o(s)=0.
\end{equation}
By using \cref{eq:h:odd,eq:h:odd:s}, we can simplify
\[
\sum_{i=k}^{2s+1-k}h_\o(i)\=h_\o(2s+1-k).
\]
When $0\le k\le s$, we have $2s+1-k\ge k$. Hence, \cref{def:h:odd,ineq:g} imply that
\[
h_\o(2s+1-k)\=g(2s+1-k,\,k)\gge 0.
\]
This completes the proof.
\end{proof}

\end{document}